\documentclass[12pt,letterpaper, leqno]{amsart}

\usepackage[latin1]{inputenc}
\usepackage[T1]{fontenc}
\usepackage{amsfonts}
\usepackage{amsmath}
\usepackage{amssymb}
\usepackage{eurosym}
\usepackage{mathrsfs}
\usepackage{palatino}

\newcommand{\R}{\mathbb{R}}

\newcommand{\N}{\mathbb{N}}
\newcommand{\No}{\mathbb{N}_0}

\numberwithin{equation}{section}

\swapnumbers
\theoremstyle{plain}
\newtheorem{thm}[equation]{Theorem}
\newtheorem{lem}[equation]{Lemma}
\newtheorem{prop}[equation]{Proposition}
\newtheorem{cor}[equation]{Corollary}

\theoremstyle{definition}
\newtheorem{defn}[equation]{Definition}

\theoremstyle{remark}
\newtheorem{rem}[equation]{Remark}

\newcommand{\pair}[2]{\langle #1,#2 \rangle}
\newcommand{\ave}[1]{\langle #1 \rangle}

\title{Multi-parameter Singular Integral Operators and Representation Theorem}

\author{Yumeng Ou}
\address{Department of Mathematics, Brown University, Providence RI, USA}
\email{yumeng\_ou@brown.edu}

\begin{document}
\maketitle

\begin{abstract}
We formulate a class of singular integral operators in arbitrarily many parameters using mixed type characterizing conditions. We also prove a multi-parameter representation theorem saying that a general operator in our class can be represented as an average of sums of dyadic shifts, which implies a new multi-parameter $T1$ theorem as a byproduct. Furthermore, an equivalence result between ours and Journ\'e's class of multi-parameter operators is established, whose proof requires the multi-parameter $T1$ theorem. These results generalize to arbitrarily many parameters recent results of Hyt\"onen \cite{Hy}, Martikainen \cite{Ma}, and Grau de la Herran \cite{Gr}.
\end{abstract}

\section{Introduction}

The study of singular integral operators on product spaces generalizing the classical Calder\'on-Zygmund theory has a history of more than 30 years, starting from \cite{FS} by Fefferman and Stein where bi-parameter operators of convolution type are carefully treated. Later, Journ\'e in \cite{Jo} established the first class of general multi-parameter singular integral operators which are not necessarily to be of convolution type, using vector valued Calder\'on-Zygmund theory and an inductive machinery. In the same paper, a multi-parameter $T1$ theorem is also proved. Very recently, Pott and Villarroya \cite{PV} formulated a new class of bi-parameter singular integral operators where the vector-valued formulations are replaced by mixed type conditions directly assumed on the operator. Their approach is then refined by Martikainen in \cite{Ma}, where he proved a bi-parameter representation of singular integrals by dyadic shifts, generalizing the famous one-parameter result of Hyt\"onen \cite{Hy}.

The representation theorem has been proven to be an incredibly useful tool in the field of singular integrals, as it enables one to reduce the problems of a general operator to problems of some very simple dyadic shift operators. For example, in \cite{HPTV} it has been utilized by Hyt\"onen, P\'erez, Treil and Volberg to obtain a simplified proof of the $A_2$ conjecture, and in \cite{DO} it has been applied to derive an upper bound estimate for iterated commutators by Dalenc and Ou. Moreover, the representation theorem also implies as a direct consequence a new $T1$ theorem. 

The theory of multi-parameter singular integral operators generally involves an additional layer of difficulty beyond the bi-parameter theory. Usually for bi-parameter problems on $\mathbb{R}\times\mathbb{R}$, in the inductive step, by slicing away one dimension one will reduce to the one-parameter setting. This is not the case for $n$-parameter problems, where $n\geq 3$. Furthermore, there are results that are true in the bi-parameter setting but fail to hold in the multi-parameter setting, for example the results regarding rectangle atoms discussed by Fefferman in \cite{Fe}. (Also see Journ\'e \cite{Jo2}.) Naturally, it has been asked by several experts in the field \cite{LP} whether one can establish a representation theorem in multi-parameters, which becomes the main motivation and the central problem this note will be dealing with. 

The first difficulty one encounters is how to generalize Martikainen's class of operators to more than two parameters, establishing a group of appropriate mixed type conditions that characterizes operators suitable to work with. Recall that in the classical $T1$ theorem, the hypotheses involve assumptions of the size and smoothness of the kernel, a weak boundedness property (WBP), and BMO conditions. It is then natural to formulate nine different so-called mixed type conditions (such as kernel/kernel, BMO/WBP and so on) for bi-parameter operators, which is, morally speaking, what Martikainen did in \cite{Ma}. However, there is no obvious way to generalize to multi-parameters formulations of such mixed type conditions. In fact, although Martikainen has done a brilliant job in \cite{Ma} to introduce the so-called full kernel and partial kernel assumptions on the operator, his assumptions are clear precisely because once a parameter is taken away, what's left becomes a one-parameter object.

The second difficulty, of course, is the proof of the representation theorem itself. Once the proper assumptions are formulated, the proof in the multi-parameter setting requires no new techniques. However, verifying that the theorem holds requires a delicately analysis of the symmetries of the operator and the particularly nice formulation of the conditions.

The main contributions of this note are the following. First, mixed type conditions for multi-parameter operators are formulated along the lines of \cite{PV} and \cite{Ma},  establishing the appropriate class of multi-parameter singular integral operators. Second, we prove a representation theorem in arbitrarily many parameters, which yields a new multi-parameter $T1$ theorem. Finally, as an application of our multi-parameter $T1$ theorem, we show that our class of multi-parameter singular integrals is equivalent to the class studied by Journ\'e in \cite{Jo}. This generalizes a recent result of Grau de la Herran \cite{Gr} to arbitrarily many parameters. This shows that Journ\'e's class of operators, originally formulated in vector-valued language, can be characterized by conditions that are more intrinsic and easier to verify.

The paper is organized as follows. In section \ref{Assump} and \ref{remark}, we define a class of multi-parameter singular integral operators characterized by new mixed type conditions. The statement of the multi-parameter representation theorem and its proof are presented in section \ref{theorem} and \ref{proof}. We then discuss the equivalence between our class and Journ\'e's class of operators in section \ref{Journe}, followed by a discussion of the necessity of some of the mixed conditions at the end.

\section*{Acknowledgement}

The author would like to thank Henri Martikainen and Jill Pipher for multiple useful conversations which granted valuable insight for the paper.

\section{A class of $n$-parameter singular integral operators}\label{Assump}

In $\mathbb{R}^{\vec{d}}:=\mathbb{R}^{d_1}\times\cdots\times\mathbb{R}^{d_n}$, where $n\in\No$ denotes the number of parameters, let $T$ be a linear operator continuously mapping $C_0^\infty(\mathbb{R}^{d_1})\times\cdots\times C^\infty_0(\mathbb{R}^{d_n})$ to its dual. $\forall S\subset \{1,2,\ldots,n\}$, define the partial adjoint $T_S$ by exchanging the $i^{th}$ variable, $\forall i\in S$, i.e. 
\[
\pair{T(f_S\otimes f_{S^c})}{g_S\otimes g_{S^c}}=\pair{T_S(g_S\otimes f_{S^c})}{f_S\otimes g_{S^c}},
\]
where $f_S, g_S$ are functions of the $i^{th}$ variables for $i\in S$, and $f_{S^c}, g_{S^c}$ are functions of the $i^{th}$ variables for $i\notin S$.

We say $T$ is in our class of \emph{$n$-parameter singular integral operators} if for any $S$, $T_S$ satisfies the following \emph{full kernel} and \emph{partial kernel assumptions}.

\subsection{Full kernel}\label{full}
For any $f=\otimes_{i=1}^n f_i,g=\otimes_{i=1}^n g_i\in C_0^\infty(\mathbb{R}^{d_1})\times\cdots\times C^\infty_0(\mathbb{R}^{d_n})$ such that $\forall i\in\{1,2,\ldots,n\}$, $\text{spt}f_i\cap\text{spt}g_i=\emptyset$, there holds
\[
\pair{T_Sf}{g}=\int_{\R^{\vec{d}}}\int_{\R^{\vec{d}}}K_S(x,y)f(y)g(x)\,dxdy,
\]
where the kernel $K_S(x,y)$ satisfies the following mixed size-H\"older conditions:

For any subset $W\subset\{1,2,\ldots,n\}$, when $|x_i-x'_i|\leq |x_i-y_i|/2, \forall i\in W$, there holds
\[
|\sum_{\Lambda\subset W}(-1)^{|\Lambda|}K^\Lambda_S(x,x';y)|\lesssim\left(\prod_{i\in W}\frac{|x_i-x'_i|^\delta}{|x_i-y_i|^{d_i+\delta}}\right)\left(\prod_{i\in\{1,2,\ldots,n\}\setminus W}\frac{1}{|x_i-y_i|^{d_i}}\right),
\]
where $0<\delta<1$ is a fixed constant, and $K^\Lambda_S(x,x';y)$ is defined as $K_S$ evaluated at $x_i$ for $i\notin \Lambda$, at $x'_i$ for $i\in \Lambda$. Note that when $W=\emptyset$, this is the pure size condition, while when $W=\{1,2,\ldots,n\}$, this becomes the H\"older condition we are familiar with in one-parameter and bi-parameter settings.

\subsection{Partial kernel}\label{partial}

Let $V$ be any nonempty proper subset of $\{1,2,\ldots,n\}$, and $f=f_V\otimes f_{V^c},g=g_V\otimes g_{V^c}\in C_0^\infty(\mathbb{R}^{d_1})\times\cdots\times C^\infty_0(\mathbb{R}^{d_n})$, where $f_V=\otimes_{i\in V}f_i$ and similarly for others. Suppose for any variable $i\in V$, $\text{spt}f_i\cap\text{spt}g_i=\emptyset$, there holds
\[
\pair{T_Sf}{g}=\int_{\otimes_{i\in V}\R^{d_i}}\int_{\otimes_{i\in V}\R^{d_i}}K^V_{S,f_{V^c},g_{V^c}}(x,y)f_V(y)g_V(x)\,dxdy,
\]
where the kernel $K^V_{S,f_{V^c},g_{V^c}}$ satisfies the following mixed size-H\"older conditions:

For any subset $W\subset V$, when $|x_i-x'_i|\leq |x_i-y_i|/2, \forall i\in W$, there holds
\[
|\sum_{\Lambda\subset W}(-1)^{|\Lambda|}K^{V,\Lambda}_{S,f_{V^c},g_{V^c}}(x,x';y)|\leq C^V_S(f_{V^c},g_{V^c})\left(\prod_{i\in W}\frac{|x_i-x'_i|^\delta}{|x_i-y_i|^{d_i+\delta}}\right)\left(\prod_{i\in V\setminus W}\frac{1}{|x_i-y_i|^{d_i}}\right),
\]
where $K^{V,\Lambda}_{S,f_{V^c},g_{V^c}}(x,x';y)$ is defined as $K^V_{S,f_{V^c},g_{V^c}}$ evaluated at $x_i$ for $i\notin \Lambda$, at $x'_i$ for $i\in \Lambda$. 

Moreover, we require that constant $C^V_S(f_{V^c},g_{V^c})$ satisfies the following WBP/BMO conditions:

For any subset $W\subset V^c$, any cubes $I_i\subset \R^{d_i}$, $i\in W$, there holds
\[
\|C^V_S(\left(\otimes_{i\in W}\chi_{I_i}\right)\otimes\left(\otimes_{i\in V^c\setminus W}1\right),\left(\otimes_{i\in W}\chi_{I_i}\right)\otimes \cdot)\|_{BMO_{prod}(\otimes_{i\in V^c\setminus W}\mathbb{R}^{d_i})}\lesssim \prod_{i\in W}|I_i|.
\]

There are several equivalent interpretations of the product BMO norm. One result proved by Pipher and Ward in \cite{PW} and reproved by Treil in \cite{Tr} is that in the multi-parameter setting, a function is in product BMO if and only if it is in dyadic product BMO uniformly with respect to any dyadic grids. Since dyadic product BMO norm can be characterized using product Carleson measure, one can express the WBP/BMO condition above by the following: For any product dyadic grid $\mathcal{D}=\otimes_{i\in V^c\setminus W}\mathcal{D}_i$,
\[
\begin{split}
&\sup\frac{1}{|\Omega|}\sum_{\substack{R\subset\Omega, R\in \mathcal{D}\\R=\otimes_{j\in V^c\setminus W}J_j}}|C^V_S(\left(\otimes_{i\in W}\chi_{I_i}\right)\otimes\left(\otimes_{i\in V^c\setminus W}1\right),\left(\otimes_{i\in W}\chi_{I_i}\right)\otimes \left(\otimes_{j\in V^c\setminus W}h_{J_j}\right))|^2 \\
&\qquad \lesssim \prod_{i\in W}|I_i|^2,
\end{split}
\]
where the supremum is taken over all the measurable open sets $\Omega$ in $\otimes_{i\in V^c\setminus W}\mathbb{R}^{d_i}$ with finite measure.

The expression above is always well defined as the functions involved are all tensor products. In the case when one can naturally extend the definition of operator $T$ to act on more general multivariate functions, one can also rephrase the WBP/BMO condition by duality as the following: For any function $h\in H^1_{prod}(\otimes_{i\in V^c\setminus W}\mathbb{R}^{d_i})$, 
\[
|C^V_S(\left(\otimes_{i\in W}\chi_{I_i}\right)\otimes\left(\otimes_{i\in V^c\setminus W}1\right),\left(\otimes_{i\in W}\chi_{I_i}\right)\otimes h)|\lesssim \left(\prod_{i\in V^c\setminus W}|I_i|\right)\|h\|_{H^1_{prod}}.
\]

This completes our definition of the \emph{$n$-parameter singular integral operators}. And one can similarly define an \emph{$n$-parameter CZO} if there are some additional boundedness assumption on the operator.

\begin{defn}
$T$ is called an $n$-parameter CZO if it is an $n$-parameter singular integral operator defined as above and $T_S:\,L^2\rightarrow L^2$, any $S\subset\{1,2,\ldots,n\}$.
\end{defn}

In order to derive the multi-parameter representation theorem for such operators later in the note, as a preparation, we will need the definition of the so called \emph{mixed BMO/WBP assumptions}, which we give as below. Note that these are not characterizing conditions of our class of singular integrals.

\subsection{BMO/WBP}\label{BMOWBP}

We say that an operator $T_S$ satisfies the mixed BMO/WBP conditions if for any subset $W\subset\{1,2,\ldots,n\}$, any cubes $I_i\subset \R^{d_i}$, $i\in W$, there holds 
\[
\|\pair{T_S(\left(\otimes_{i\in W}\chi_{I_i}\right)\otimes\left(\otimes_{i\in W^c}1\right))}{\left(\otimes_{i\in W}\chi_{I_i}\right)\otimes \cdot}\|_{BMO_{prod}(\otimes_{i\in W^c}\mathbb{R}^{d_i})}\lesssim \prod_{i\in W}|I_i|.
\]

This is the pure BMO condition when $W=\emptyset$,  and the pure dyadic weak boundedness property when $W=\{1,2,\ldots,n\}$. Again, one can interpret the product BMO norm in several different ways, as we described above.

To end the section, we would like to emphasize that the class of singular integral operators defined above is indeed a generalization of the most natural classes of one-parameter and bi-parameter singular integral operators studied in harmonic analysis. When $n=1$, it coincides with the class of singular integral operators associated with standard kernel. When $n=2$, it is the same as the class of bi-parameter operators defined by Martikainen in \cite{Ma} (modulo that some of the conditions in partial kernel assumptions are formulated slightly differently), and is known to be equivalent to the classes of Journ\'e \cite{Jo} and Pott-Villarroya \cite{PV}, a result recently proved by Grau de la Herran \cite{Gr}. 

Furthermore, it is not hard to examine that our class of $n$-parameter singular integrals includes operators of tensor product type as a special case. Let's take a look at the case $n=3$ as an example. Given CZOs $T_i$ defined on $\mathbb{R}^{d_i}$, $i=1,2,3$, it is easy to see that the operator $T_1\otimes T_2\otimes T_3$ satisfies the full kernel assumptions. To check one of the partial kernel assumptions, for any test functions with $\text{spt}f_1\cap\text{spt}g_1=\emptyset$, one can define a partial kernel 
\[
K^{\{1\}}_{f_2\otimes f_3,g_2\otimes g_3}(x_1,y_1)=K_1(x_1,y_1)\pair{T_2\otimes T_3(f_2\otimes f_3)}{g_2\otimes g_3},
\]
where $K_1(x_1,y_1)$ is the kernel of $T_1$. Observe that $T_2\otimes T_3$ is a Journ\'e type bi-parameter CZO studied in \cite{Jo}, hence is bounded on $L^2$ and maps $1\otimes 1$ into product BMO, which thus implies the required WBP/BMO conditions for constants $C^{\{1\}}(f_2\otimes f_3,g_2\otimes g_3)$. We will give a more thorough discussion of the Journ\'e type multi-parameter singular integral operators in section \ref{Journe}.

\section{A remark on the well-definedness of the BMO assumptions}\label{remark}

Among the various conditions satisfied by an $n$-parameter operator $T$, many of them are establishing certain bounds on pairings involving $T$ acting on function $1$ in some of the variables. It is thus necessary to articulate how these objects are defined. For simplicity, let's look at the case $n=3$.

Recall that in the partial kernel assumptions, if $f=f_{1}\otimes f_2\otimes f_3, g=g_{1}\otimes g_2\otimes g_3$,  and $\text{spt}f_{1}\cap\text{spt}g_{1}=\text{spt}f_{2}\cap\text{spt}g_{2}=\emptyset$ (i.e. $V=\{1,2\}$), one wants to show that $C_S^V(1,\cdot)\in BMO(\R^{d_3})$, which according to \cite{PW} is the same as showing that for any dyadic system $\mathcal{D}$ of $\R^{d_3}$, it is in dyadic $BMO_{\mathcal{D}}(\R^{d_3})$.

Hence, it suffices to give a meaning to $C_S^V(1,h_{I_3})$ for any Haar function in the third variable, i.e. to define the pairing $\pair{T_S(f_{1}\otimes f_2\otimes 1)}{g_{1}\otimes g_2\otimes h_{I_3}}$. This can be done by dividing $1=\chi_{3I_3}+\chi_{(3I_3)^c}$, where the first term makes sense since $T$ is continuous (more precisely, one needs kernel representation, WBP and dominated convergence to justify the well-definedness of the bilinear form of non smooth functions), while the second term can be defined using the full kernel representation whose convergence is guaranteed by H\"older condition.

Second, still in the partial kernel assumptions, if one only has $f_3\cap g_3=\emptyset$ (i.e. $V=\{3\}$), the well-definedness of constant $C_S^V(\chi_{I_1}\otimes 1,\chi_{I_1}\otimes\cdot)$ is similar as the above case, so we only look at the meaning of $C_S^V(1\otimes 1,\cdot)$ as a function in dyadic $BMO_{\mathcal{D}(\R^{d_1}\times\R^{d_2})}$. To define $\pair{T_S(1\otimes 1\otimes f_3)}{h_{I_1}\otimes h_{I_2}\otimes g_3}$, clearly, one can divide $1\otimes 1=\chi_{3I_1}\otimes\chi_{3I_2}+\chi_{3I_1}\otimes\chi_{(3I_2)^c}+\chi_{(3I_1)^c}\otimes\chi_{3I_2}+\chi_{(3I_1)^c}\otimes\chi_{(3I_2)^c}$, where the first and last term are easy to deal with. While for the mixed terms, say, the third one, if $\chi_{(3I_1)^c}$ is replaced by a $C_0^\infty$ function, then the pairing is apparently well defined through the partial kernel representation. Now even though $\chi_{(3I_1)^c}$ is only bounded, we can still define the pairing as
\[
\int K^{\{1,3\}}_{S,\chi_{3I_2},h_{I_2}}(x_1,y_1,x_3,y_3)\chi_{(3I_1)^c}(y_1)f_3(y_3)h_{I_1}(x_1)g_3(x_3)\,dx_1dx_3dy_1dy_3,
\]
where the integral converges since one can change the kernel to
\[
K^{\{1,3\}}_{S,\chi_{3I_2},h_{I_2}}(x_1,y_1,x_3,y_3)-K^{\{1,3\}}_{S,\chi_{3I_2},h_{I_2}}(x_1,y_1,c_{I_3},y_3)
\]
and use the mixed H\"older-size condition.

Finally, in the BMO/WBP assumptions, to give a meaning to
\[
\pair{T_S(\left(\otimes_{i\in W}\chi_{I_i}\right)\otimes\left(\otimes_{i\in W^c}1\right))}{\left(\otimes_{i\in W}\chi_{I_i}\right)\otimes\cdot},
\]
it is then sufficient to define what it means for the function to be paired with tensors of Haar functions. This can be done by dividing $1\otimes\cdots\otimes 1$ into several parts  similarly as above, and use partial kernel representation and H\"older conditions to obtain the convergence of the corresponding integrals.

\section{Multi-parameter representation theorem}\label{theorem}

In order to formulate the representation theorem in the multi-parameter setting, one first needs to recall the notion of shifted dyadic grids, which are essential elements of the theorem. Denote $\mathcal{D}^0_i:=\{2^{-k}([0,1]^{d_i}+m):\,k\in\mathbb{Z},m\in\mathbb{Z}^{d_i}\}$ as the standard dyadic grid in the $i$-th variable, $1\leq i\leq n$. Let $\omega=(\omega^j_i)_{j\in\mathbb{Z}}\in(\{0,1\}^{d_i})^{\mathbb{Z}}$ and $I\dotplus \omega_i:=I+\sum_{j:2^{-j}<\ell(I)}2^{-j}\omega^j_i$, then
\[
\mathcal{D}^\omega_i:=\{I\dotplus\omega_i:\,I\in\mathcal{D}^0_i\}
\]
is a shifted dyadic grid associated with parameter $\omega_i$. We usually write $\mathcal{D}_i$ for short in practice when the dependence on $\omega_i$ is not explicitly needed.

If we assume each $\omega_i$ is an independent random variable having an equal probability $2^{-d_i}$ of taking any of the $2^{d_i}$ values in $\{0,1\}^{d_i}$, we obtain a random dyadic system $\mathcal{D}_1\times\cdots\times\mathcal{D}_n$.

A \emph{dyadic shift} with parameter $i_1,j_1,\ldots,i_n,j_n\in\mathbb{N}$ associated with dyadic grids $\mathcal{D}_1,\ldots,\mathcal{D}_n$ is an $L^2\rightarrow L^2$ operator with norm $\leq 1$ defined as
\[
\begin{split}
&S^{i_1j_1,\ldots,i_nj_n}_{\mathcal{D}_1\ldots\mathcal{D}_n}f\\
&:=\sum_{s=1}^n\sum_{K_s\in\mathcal{D}_s}\sum_{\substack{I_s,J_s\in\mathcal{D}_s,I_s,J_s\subset K_s\\\ell(I_s)=2^{-i_s}\ell(K_s)\\\ell(J_s)=2^{-j_s}\ell(K_s)}}a_{I_1J_1K_1\ldots I_nJ_nK_n}\pair{f}{h_{I_1}\otimes\cdots\otimes h_{I_n}}h_{J_1}\otimes\cdots\otimes h_{J_n}\\
&=:\sum_{s=1}^n\sum_{K_s\in\mathcal{D}_s}\sum_{\substack{I_s,J_s\in\mathcal{D}_s\\I_s,J_s\subset K_s}}^{(i_s,j_s)}a_{I_1J_1K_1\ldots I_nJ_nK_n}\pair{f}{h_{I_1}\otimes\cdots\otimes h_{I_n}}h_{J_1}\otimes\cdots\otimes h_{J_n},
\end{split}
\]
where the coefficients satisfy
\[
|a_{I_1J_1K_1\ldots I_nJ_nK_n}|\leq \frac{\sqrt{|I_1||J_1|\cdots |I_n||J_n|}}{|K_1|\cdots |K_n|},
\]
and $h_{I_s}$ is a Haar function on $I_s$, similarly for $h_{J_s}$. Note that for any dyadic cube $I\subset \mathbb{R}^{d_i}$, there are $2^{d_i}$ associated Haar functions $h_{I}$, with one of them being the noncancellative function $|I|^{-1/2}\chi_I$ and all the other ones being cancellative. We allow any choices of Haar functions, noncancellative or cancellative, in the definition of dyadic shifts. In addition, we will call the dyadic shift \emph{cancellative} if all the Haar functions that appear in the sum are cancellative. It is not hard to show that when the shift is cancellative, the $L^2$ boundedness requirement in fact follows from the boundedness of coefficients directly. Furthermore, it is also worth observing that $n$-parameter dyadic paraproducts are particular examples of noncancellative dyadic shifts.

Now we are ready to state the representation theorem. Recall that $T$ is said to be an $n$-parameter singular integral operator in our class if it satisfies both the full kernel and partial kernel assumptions defined in section \ref{full}, \ref{partial}.

\begin{thm}\label{Repre}
For an $n$-parameter singular integral operator $T$, which satisfies in addition the BMO/WBP assumptions (see section \ref{BMOWBP}), there holds for some $n$-parameter shifts $S^{i_1j_1\ldots i_nj_n}_{\mathcal{D}_{1}\ldots\mathcal{D}_{n}}$ that
\[
\pair{Tf}{g}=C_T\mathbb{E}_{\omega_1}\mathbb{E}_{\omega_2}\cdots\mathbb{E}_{\omega_n}\sum_{s=1}^n\sum_{(i_s,j_s)\in\N^2}\left(\prod_{t=1}^n 2^{-\max(i_t,j_t)\delta/2}\right)\pair{S^{i_1j_1\ldots i_nj_n}_{\mathcal{D}_{1}\ldots\mathcal{D}_{n}}f}{g},
\]
where noncancellative shifts may only appear when there is some $s$ such that $(i_s,j_s)=(0,0)$.
\end{thm}

$f$ and $g$ above are arbitrary functions taken from some particularly nice dense subset of $L^2(\mathbb{R}^{\vec{d}})$, for example, the finite linear combinations of tensor products of univariate functions in $C^\infty_0(\mathbb{R}^{d_i})$. Hence, according to the uniform boundedness of dyadic shifts, an immediate result implied by the representation theorem is the following.

\begin{cor}\label{Cor}
An $n$-parameter singular integral operator $T$ satisfying the BMO/WBP assumptions is bounded on $L^2(\R^{\vec{d}})$.
\end{cor}

\begin{rem}
In the one-parameter and bi-parameter versions of the representation theorem, see \cite{Hy}, \cite{Ma}, one needs the additional a priori assumption that $T$ is bounded on $L^2$ in order to justify the convergence of some infinite series in the proof. This makes the $T1$ type corollary only a quantitative result. However, very recently, it is suggested by T. Hyt\"onen that one can prove the representation theorem without assuming any a priori bound on $T$, by first proving a "weak representation" depending on functions $f, g$, which then implies that $T$ is bounded on $L^2$. Hence, the corollary obtained above is indeed a $T1$ theorem of full strength, which is certainly of its own interest. Previously, the only known $T1$ type theorem in more than two parameters is proved by Journ\'e in \cite{Jo} by induction, using a vector valued argument. The advantage of our $T1$ theorem is that the mixed type conditions are expressed in a more transparent way and much easier to verify. In fact, we will see an application of our $T1$ theorem later in the paper, when we establish the relationship between Journ\'e's and our class of multi-parameter singular integral operators.
\end{rem}

Another useful observation is that due to the symmetry of the assumptions on the $n$-parameter singular integral operators, one can conclude that if $T$ is an $n$-parameter SIO satisfying the BMO/WBP assumptions, then any of its partial adjoints $T_S$ is bounded on $L^2$. Hence $T$ is an $n$-parameter CZO defined in section \ref{Assump}. In fact, the other direction also holds true, i.e. $T$ being an $n$-parameter CZO implies the BMO/WBP assumptions. We leave the discussion of this point to the end of the paper.

\section{Proof of Theorem \ref{Repre}}\label{proof}

Let's prove the case $n=3$ as an example, which is sufficient in showing the new difficulties arising in the multi-parameter setting and in explaining our strategy. Roughly speaking, we will first establish a tri-parameter version of the averaging formula, where the notions of good and bad cubes appear. Then, by decomposing the pairing $\pair{Tf}{g}$ into several mixed parts (separated, inside, near and equal), a case by case discussion will lead to the desired result.

\subsection{Randomizing process and averaging formula}

To start with, through a similar process of randomization independently in each variable, as described in \cite{Hy} and \cite{Ma}, it is not hard to obtain the following tri-parameter version of the key averaging formula:
\[
\begin{split}
\pair{Tf}{g}=&C\mathbb{E}\sum_{I_1,J_1\in\mathcal{D}_{1}}\sum_{I_2,J_2\in\mathcal{D}_{2}}\sum_{I_3,J_3\in\mathcal{D}_{3}}\chi_{\text{good}}(sm(I_1,J_1))\chi_{\text{good}}(sm(I_2,J_2))\chi_{\text{good}}(sm(I_3,J_3))\\
&\pair{T(h_{I_1}\otimes h_{I_2}\otimes h_{I_3})}{h_{J_1}\otimes h_{J_2}\otimes h_{J_3}}\pair{f}{h_{I_1}\otimes h_{I_2}\otimes h_{I_3}}\pair{g}{h_{J_1}\otimes h_{J_2}\otimes h_{J_3}},
\end{split}
\]
where $\mathbb{E}=\mathbb{E}_{\omega_{1}}\mathbb{E}_{\omega_{2}}\mathbb{E}_{\omega_{3}}$ and $C=1/(\pi_{\text{good}}^{1}\pi_{\text{good}}^{2}\pi_{\text{good}}^{3})$. 

We remind the readers that a cube $I_i\in\mathcal{D}_i$ is called \emph{bad} if there is another $\tilde{I_i}\in\mathcal{D}_i$ such that $\ell(\tilde{I_i})\geq 2^r\ell(I_i)$ and $d(I_i,\partial\tilde{I_i})\leq 2\ell(I_i)^{\gamma_i}\ell(\tilde{I_i})^{1-\gamma_i}$, where $r$ is a fixed large number, $\gamma_i:=\delta/(2d_i+2\delta)$, and $\delta$ is the constant that appears in the kernel assumptions of the operator. Naturally, a cube is called \emph{good} if it is not bad. And $\pi_{\text{good}}^i:=\mathbb{P}_{\omega_i}(I_i\dotplus\omega_i\,\text{is good})$ is a parameter depending only on $\delta$, $d_i$ and $r$. One always fixes an $r$ large enough so that $\pi_{\text{good}}^i>0$ for any $1\leq i\leq n$.

In order to show the desired representation, we will then split the sums on the right hand side of the averaging formula into several pieces depending on the relative sizes of $I_i, J_i$, $i=1,2,3$, and whether the smaller cubes are far away, strictly inside, exactly equal, or close to the larger cubes (i.e. Separated, Inside, Equal or Near). More specifically, for each variable $i$, we split the sum
\[
\sum_{I_i}\sum_{J_i}=\sum_{\ell(I_i)\leq \ell(J_i)}+\sum_{\ell(I_i)>\ell(J_i)}=:I+II.
\]
Then decompose
\[
\begin{split}
I&=\sum_{\substack{\ell(I_i)\leq\ell(J_i)\\ d(I_i,J_i)>\ell(I_i)^{\gamma_i}\ell(J_i)^{1-\gamma_i}}}+\sum_{I_i\subsetneq J_i}+\sum_{I_i=J_i}+\sum_{\substack{\ell(I_i)\leq\ell(J_i)\\ d(I_i,J_i)\leq \ell(I_i)^{\gamma_i}\ell(J_i)^{1-\gamma_i}\\ I_i\cap J_i=\emptyset}}\\
&=:\text{Separated}+\text{Inside}+\text{Equal}+\text{Near}
\end{split}
\]
and similarly for $II$. The strategy is to prove that each of the terms above can be represented as sums of dyadic shifts.

Many of the cases can be discussed using the same technique as in \cite{Ma}, while for some mixed cases, new multi-parameter phenomena may appear and require extreme care. The good news is that the new mixed cases won't do us much harm since we have already formulated the proper assumptions on the operators at the beginning to handle them.

As one has already encountered in the bi-parameter setting in \cite{Ma}, different types of mixed paraproducts will appear according to the relative sizes of $I_i, J_i$. Since the worst situations one can expect are the mixed cases, we will look at the part of the sum corresponding to $|I_1|\leq |J_1|,|I_2|\leq |J_2|,|I_3|>|J_3|$, observing that other cases are symmetric or even simpler. According to the averaging formula, it thus suffices to assume that $I_1,I_2,J_3$ are all good cubes.

Moreover, recall that in \cite{Hy} and \cite{Ma}, the Separated, Near, and Equal parts of the sum can basically be estimated using full kernel assumptions and WBP, while the Inside part, being the most difficult one, involves in addition all the BMO type estimates. Hence, we will study the Inside/Inside/Inside part next, where all the new multi-parameter phenomena will appear. Note that although this is only one of the many cases one needs to discuss in order to obtain a full proof of Theorem \ref{Repre}, all the main difficulties in other cases are in fact already embedded in Inside/Inside/Inside, a fact that will become more and more clear throughout the proof. We want to emphasize that the reason why we assumed from the beginning that all the assumptions hold true for any partial adjoint $T_S$ of $T$ is exactly because of the symmetry of the mixed cases.

\subsection{Inside/Inside/Inside}

In this section, we study the case Inside/Inside/Inside, i.e. the summation over $I_1\subsetneq J_1,I_2\subsetneq J_2, J_3\subsetneq I_3$. Recall that $I_1, I_2, J_3$ are all good cubes. One first decomposes
\[
\pair{T(h_{I_1}\otimes h_{I_2}\otimes h_{I_3})}{h_{J_1}\otimes h_{J_2}\otimes h_{J_3}}=I+II+III+IV+V+VI+VII+VIII,
\]
where
\[
\begin{split}
&I:=\pair{T(h_{I_1}\otimes h_{I_2}\otimes s_{J_3I_3})}{s_{I_1J_1}\otimes s_{I_2J_2}\otimes h_{J_3}},\\
&II:=\ave{h_{I_3}}_{J_3}\pair{T(h_{I_1}\otimes h_{I_2}\otimes 1)}{s_{I_1J_1}\otimes s_{I_2J_2}\otimes h_{J_3}},\\
&III:=\ave{h_{J_2}}_{I_2}\pair{T(h_{I_1}\otimes h_{I_2}\otimes s_{J_3I_3})}{s_{I_1J_1}\otimes 1\otimes h_{J_3}},\\
&IV:=\ave{h_{J_2}}_{I_2}\ave{h_{I_3}}_{J_3}\pair{T(h_{I_1}\otimes h_{I_2}\otimes 1)}{s_{I_1J_1}\otimes 1\otimes h_{J_3}},\\
&V:=\ave{h_{J_1}}_{I_1}\pair{T(h_{I_1}\otimes h_{I_2}\otimes s_{J_3I_3})}{1\otimes s_{I_2J_2}\otimes h_{J_3}},\\
&VI:=\ave{h_{J_1}}_{I_1}\ave{h_{I_3}}_{J_3}\pair{T(h_{I_1}\otimes h_{I_2}\otimes 1)}{1\otimes s_{I_2J_2}\otimes h_{J_3}},\\
&VII:=\ave{h_{J_1}}_{I_1}\ave{h_{J_2}}_{I_2}\pair{T(h_{I_1}\otimes h_{I_2}\otimes s_{J_3I_3})}{1\otimes 1\otimes h_{J_3}},\\
&VIII:=\ave{h_{J_1}}_{I_1}\ave{h_{J_2}}_{I_2}\ave{h_{I_3}}_{J_3}\pair{T(h_{I_1}\otimes h_{I_2}\otimes 1)}{1\otimes 1\otimes h_{J_3}}.
\end{split}
\]

In the above, $s_{I_1J_1}:=\chi_{Q_1^c}(h_{J_1}-\ave{h_{J_1}}_{Q_1})$, $s_{I_2J_2}:=\chi_{Q_2^c}(h_{J_2}-\ave{h_{J_2}}_{Q_2})$, $Q_1,Q_2$ being the child of $J_1,J_2$ containing $I_1, I_2$, respectively, and $s_{J_3I_3}:=\chi_{Q_3^c}(h_{I_3}-\ave{h_{I_3}}_{Q_3})$, $Q_3$ being the child of $I_3$ containing $J_3$. The relevant properties are $\text{spt}s_{I_1J_1}\subset Q_1^c$, $\text{spt}s_{I_2J_2}\subset Q_2^c$, $\text{spt}s_{J_3I_3}\subset Q_3^c$, and $|s_{I_1J_1}|\leq 2|J_1|^{-1/2}$, $|s_{I_2J_2}|\leq 2|J_2|^{-1/2}$, $|s_{J_3I_3}|\leq 2|I_3|^{-1/2}$.

Next, we show that the sum corresponding to each of the eight terms above can be realized as a sum of dyadic shifts. The estimate of  term I doesn't require any BMO conditions, while all the other terms require delicate BMO norm estimates and boundedness results of paraproducts. More specifically, we will use one-parameter paraproduct to analyze term III, V, II, bi-parameter paraproduct for term IV, VI, VII, and tri-parameter paraproduct for the last term VIII. The reader will easily see that when the number of parameters is more than three, analogous argument can be established.

\subsubsection{Term I}
As the functions in the pairing are all disjointly supported, following from the full kernel assumptions, one can argue similarly as in \cite{Ma} Lemma $7.1$ that there holds
\[
\begin{split}
&|\pair{T(h_{I_1}\otimes h_{I_2}\otimes s_{J_3I_3})}{s_{I_1J_1}\otimes s_{I_2J_2}\otimes h_{J_3}}|\\
&\quad \lesssim \frac{|I_1|^{1/2}}{|J_1|^{1/2}}\left(\frac{\ell(I_1)}{\ell(J_1)}\right)^{\delta/2}\frac{|I_2|^{1/2}}{|J_2|^{1/2}}\left(\frac{\ell(I_2)}{\ell(J_2)}\right)^{\delta/2}\frac{|J_3|^{1/2}}{|I_3|^{1/2}}\left(\frac{\ell(J_3)}{\ell(I_3)}\right)^{\delta/2}.
\end{split}
\]

We omit the details. Hence, term I can be realized in the form
\[
C\sum_{i_1=1}^\infty\sum_{i_2=1}^\infty\sum_{j_3=1}^\infty 2^{-i_1\delta/2}2^{-i_2\delta/2}2^{-j_3\delta/2}\pair{S^{i_10i_200j_3}f}{g}.
\]

\subsubsection{Term III, V, II}
Next we deal with term III (symmetric with term V) which can be written in the form
\[
\begin{split}
&\sum_{I_1\subsetneq J_1}\sum_{J_3\subsetneq I_3}\sum_{I_2\subsetneq J_2}\ave{h_{J_2}}_{I_2}\pair{T(h_{I_1}\otimes h_{I_2}\otimes s_{J_3I_3})}{s_{I_1J_1}\otimes 1\otimes h_{J_3}}\cdot\\
&\qquad\qquad\pair{f}{h_{I_1}\otimes h_{I_2}\otimes h_{I_3}}\pair{g}{h_{J_1}\otimes h_{J_2}\otimes h_{J_3}}\\
&=\sum_{I_1\subsetneq J_1}\sum_{J_3\subsetneq I_3}\sum_V\ave{\pair{g}{h_{J_1}\otimes h_{J_3}}_{1,3}}_V\pair{T(h_{I_1}\otimes h_{V}\otimes s_{J_3I_3})}{s_{I_1J_1}\otimes 1\otimes h_{J_3}}\cdot\\
&\qquad\qquad\pair{f}{h_{I_1}\otimes h_V\otimes h_{I_3}}.
\end{split}
\] 

It is not hard to show the correct normalization of the coefficient
\[
|\pair{T(h_{I_1}\otimes h_{V}\otimes s_{J_3I_3})}{s_{I_1J_1}\otimes 1\otimes h_{J_3}}|\lesssim\frac{|I_1|^{1/2}}{|J_1|^{1/2}}\left(\frac{\ell(I_1)}{\ell(J_1)}\right)^{\delta/2}\frac{|J_3|^{1/2}}{|I_3|^{1/2}}\left(\frac{\ell(J_3)}{\ell(I_3)}\right)^{\delta/2}|V|^{1/2},
\]
which means that term III can be realized in the form
\[
C\sum_{i_1=1}^\infty\sum_{j_3=1}^\infty 2^{-i_1\delta/2}2^{-j_3\delta/2}\pair{S^{i_10000j_3}f}{g}.
\]

As $S^{i_10000j_3}$ is a noncancellative shift, we need to show its boundedness separately, which requires a one-parameter BMO type estimate. Rewrite
\[
\begin{split}
&\sum_V\ave{\pair{g}{h_{J_1}\otimes h_{J_3}}_{1,3}}_V\pair{T(h_{I_1}\otimes h_{V}\otimes s_{J_3I_3})}{s_{I_1J_1}\otimes 1\otimes h_{J_3}}\pair{f}{h_{I_1}\otimes h_V\otimes h_{I_3}}\\
&=\sum_V\ave{\pair{g}{h_{J_1}\otimes h_{J_3}}_{1,3}}_V\pair{\pair{T^*(s_{I_1J_1}\otimes 1\otimes h_{J_3})}{h_{I_1}\otimes s_{J_3I_3}}_{1,3}}{h_V}_2\cdot\\
&\qquad\qquad\pair{\pair{f}{h_{I_1}\otimes h_{I_3}}_{1,3}}{h_V}_2\\
&=:C2^{-i_1\delta/2}2^{-j_3\delta/2}\pair{\pair{f}{h_{I_1}\otimes h_{I_3}}_{1,3}}{\Pi_{b_{I_1J_1J_3I_3}}(\pair{g}{h_{J_1}\otimes h_{J_3}}_{1,3})}_2\\
&=C2^{-i_1\delta/2}2^{-j_3\delta/2}\pair{h_{J_1}\otimes \Pi^*_{b_{I_1J_1J_3I_3}}(\pair{f}{h_{I_1}\otimes h_{I_3}}_{1,3})\otimes h_{J_3}}{g},
\end{split}
\]
where $b_{I_1J_1J_3I_3}=\pair{T^*(s_{I_1J_1}\otimes 1\otimes h_{J_3})}{h_{I_1}\otimes s_{J_3I_3}}_{1,3}/(C2^{-i_1\delta/2}2^{-j_3\delta/2})$, and $\Pi_{a}$ denotes a one-parameter paraproduct in the second variable defined as
\[
\Pi_b(f)(x_2)=\sum_V\pair{b}{h_V}_2\pair{f}{|V|^{-1/2}\chi_V}_2h_V(x_2)|V|^{-1/2}.
\]

Hence, one has
\[
\begin{split}
S^{i_10000j_3}f&=\sum_{J_1}\sum_{\substack{I_1\subset J_1\\\ell(I_1)=2^{-i_1}\ell(J_1)}}\sum_{I_3}\sum_{\substack{J_3\subset I_3\\\ell(J_3)=2^{-j_3}\ell(I_3)}}h_{J_1}\otimes \Pi^*_{b_{I_1J_1J_3I_3}}(\pair{f}{h_{I_1}\otimes h_{I_3}}_{1,3})\otimes h_{J_3}\\
&=:\sum_{J_1}\sum_{I_1\subset J_1}^{(i_1)}\sum_{I_3}\sum_{J_3\subset I_3}^{(j_3)}h_{J_1}\otimes \Pi^*_{b_{I_1J_1J_3I_3}}(\pair{f}{h_{I_1}\otimes h_{I_3}}_{1,3})\otimes h_{J_3}.
\end{split}
\]

One first obverses that there holds the following estimate:

\begin{lem}\label{oneparaBMO}
$\|b_{I_1J_1J_3I_3}\|_{BMO(\R^{d_2})}\lesssim \frac{|I_1|^{1/2}}{|J_1|^{1/2}}\frac{|J_3|^{1/2}}{|I_3|^{1/2}}.$
\end{lem}

\begin{proof}
For any cube $V$ in $\R^{d_2}$, let $a$ be a function on $\R^{d_2}$ with $\text{spt}a\subset V$, $|a|\leq 1$ and $\int a=0$. It suffices to show that
\[
|\pair{T(h_{I_1}\otimes a\otimes s_{J_3I_3})}{s_{I_1J_1}\otimes 1\otimes h_{J_3}}|\lesssim \frac{|I_1|^{1/2}}{|J_1|^{1/2}}\frac{|J_3|^{1/2}}{|I_3|^{1/2}}\left(\frac{\ell(I_1)}{\ell(J_1)}\right)^{\delta/2}\left(\frac{\ell(J_3)}{\ell(I_3)}\right)^{\delta/2}|V|.
\]

Since in the pairing, functions of the first and third variables are disjointly supported, one can use partial kernel representation, the standard kernel estimate of $K_{a,1}^{\{1,3\}}$ and boundedness of constant $C^{\{1,3\}}(a,1)$ to derive the desired estimate. We omit the details.
\end{proof}

This then implies that $\Pi^*_{b_{I_1J_1J_3I_3}}$ is bounded on $L^2(\mathbb{R}^{d_2})$ with norm bounded by $(|I_1|/|J_1|)^{1/2}(|J_3|/|I_3|)^{1/2}$. We now claim that $\|S^{i_10000j_3}f\|_2\lesssim\|f\|_2$. The idea behind is similar to Proposition $4.5$ in \cite{Ma}, but what we face here is more complicated as the relative sizes of cubes in different variables are of mixed type. 

\begin{prop}\label{bdd}
For arbitrary $i_1, j_3$, there holds
\[
\|\sum_{J_1}\sum_{I_1\subset J_1}^{(i_1)}\sum_{I_3}\sum_{J_3\subset I_3}^{(j_3)}h_{J_1}\otimes \Pi^*_{b_{I_1J_1J_3I_3}}(\pair{f}{h_{I_1}\otimes h_{I_3}}_{1,3})\otimes h_{J_3}\|^2_{L^2(\mathbb{R}^{\vec{d}})}\lesssim \|f\|^2_{L^2(\mathbb{R}^{\vec{d}})}.
\]
\end{prop}

\begin{proof}
The orthogonality of Haar systems implies that
\[
\begin{split}
&\|\sum_{J_1}\sum_{I_1\subset J_1}^{(i_1)}\sum_{I_3}\sum_{J_3\subset I_3}^{(j_3)}h_{J_1}\otimes \Pi^*_{b_{I_1J_1J_3I_3}}(\pair{f}{h_{I_1}\otimes h_{I_3}}_{1,3})\otimes h_{J_3}\|^2_{L^2(\mathbb{R}^{\vec{d}})}\\
&=\sum_{J_1}\sum_{J_3}\|\sum_{I_1\subset J_1}^{(i_1)}\Pi^*_{b_{I_1J_1J_3J_3^{(j_3)}}}(\pair{f}{h_{I_1}\otimes h_{J_3^{(j_3)}}}_{1,3})\|^2_{L^2(\mathbb{R}^{d_2})}\\
&\leq\sum_{J_1}\sum_{J_3}\left(\sum_{I_1\subset J_1}^{(i_1)}\|\Pi^*_{b_{I_1J_1J_3J_3^{(j_3)}}}(\pair{f}{h_{I_1}\otimes h_{J_3^{(j_3)}}}_{1,3})\|_{L^2(\mathbb{R}^{d_2})}\right)^2,
\end{split}
\]
where $J_3^{(j_3)}$ denotes the $j_3$-th dyadic ancestor of $J_3$. Now let $P_{J_1}^{i_1}$ denote the orthogonal projection from $L^2(\mathbb{R}^{d_1})$ onto the span of $\{h_{I_1}:\,I_1\subset J_1, \,\ell(I_1)=2^{-i_1}\ell(J_1)\}$, thus,
\[
\begin{split}
&\|\Pi^*_{b_{I_1J_1J_3J_3^{(j_3)}}}(\pair{f}{h_{I_1}\otimes h_{J_3^{(j_3)}}}_{1,3})\|_{L^2(\mathbb{R}^{d_2})}\lesssim\frac{|I_1|^{1/2}}{|J_1|^{1/2}}\frac{|J_3|^{1/2}}{|J_3^{(j_3)}|^{1/2}}\|\pair{f}{h_{I_1}\otimes h_{J_3^{(j_3)}}}_{1,3}\|_{L^2(\mathbb{R}^{d_2})}\\
&\leq\frac{|I_1|^{1/2}}{|J_1|^{1/2}}\frac{|J_3|^{1/2}}{|J_3^{(j_3)}|^{1/2}}\left(\int_{\mathbb{R}^{d_2}}\int_{I_1}|P_{J_1}^{i_1}(\pair{f}{h_{J_3^{(j_3)}}}_3)|^2\,dx_1dx_2\right)^{1/2}.
\end{split}
\]

Therefore, one has
\[
\begin{split}
&\|\sum_{J_1}\sum_{I_1\subset J_1}^{(i_1)}\sum_{I_3}\sum_{J_3\subset I_3}^{(j_3)}h_{J_1}\otimes \Pi^*_{b_{I_1J_1J_3I_3}}(\pair{f}{h_{I_1}\otimes h_{I_3}}_{1,3})\otimes h_{J_3}\|^2_{L^2(\mathbb{R}^{\vec{d}})}\\
&\lesssim \sum_{J_1}\sum_{J_3}\left(\sum_{I_1\subset J_1}^{(i_1)}\frac{|I_1|^{1/2}}{|J_1|^{1/2}}\frac{|J_3|^{1/2}}{|J_3^{(j_3)}|^{1/2}}\left(\int_{\mathbb{R}^{d_2}}\int_{I_1}|P_{J_1}^{i_1}(\pair{f}{h_{J_3^{(j_3)}}}_3)|^2\,dx_1dx_2\right)^{1/2}\right)^2,
\end{split}
\]
which by H\"older's inequality is bounded by
\[
\begin{split}
&\lesssim \sum_{J_1}\sum_{J_3}\left(\sum_{I_1\subset J_1}^{(i_1)}\frac{|I_1|}{|J_1|}\frac{|J_3|}{|J_3^{(j_3)}|}\right)\left(\sum_{I_1\subset J_1}^{(i_1)}\int_{\mathbb{R}^{d_2}}\int_{I_1}|P_{J_1}^{i_1}(\pair{f}{h_{J_3^{(j_3)}}}_3)|^2\,dx_1dx_2\right)\\
&=\sum_{J_3}\frac{|J_3|}{|J_3^{(j_3)}|}\sum_{J_1}\int_{\mathbb{R}^{d_2}}\int_{\mathbb{R}^{d_1}}|P_{J_1}^{i_1}(\pair{f}{h_{J_3^{(j_3)}}}_3)|^2\,dx_1dx_2\\
&=\sum_{J_3}\frac{|J_3|}{|J_3^{(j_3)}|}\|\pair{f}{h_{J_3^{(j_3)}}}_3\|^2_{L^2(\mathbb{R}^{d_1+d_2})},
\end{split}
\]
where the last step above follows from the orthogonality of $\{P_{J_1}^{i_1}\}_{J_1}$. Note that by reindexing $J_3^{(j_3)}$ as $I_3$, the RHS can be written as
\[
\sum_{I_3}\sum_{J_3\subset I_3}^{(j_3)}\frac{|J_3|}{|I_3|}\|\pair{f}{h_{I_3}}_3\|^2_{L^2(\mathbb{R}^{d_1+d_2})}=\|f\|^2_{L^2(\mathbb{R}^{\vec{d}})},
\]
which completes the proof.
\end{proof}

This finishes the discussion of term III. Though term II is not completely symmetric to III or V, it can be handled similarly by realized in a form of sums of terms involving one-parameter paraproducts and by using the following BMO lemma. The boundedness of the arising dyadic shifts then follows from a similar argument as Proposition \ref{bdd}. 

\begin{lem}
Define $b_{I_1J_1I_2J_2}=\pair{T(h_{I_1}\otimes h_{I_2}\otimes 1)}{s_{I_1J_1}\otimes s_{I_2J_2}}_{1,2}/(C2^{-i_1\delta/2}2^{-i_2\delta/2})$, then
\[
\|b_{I_1J_1I_2J_2}\|_{BMO(\R^{d_3})}\lesssim \frac{|I_1|^{1/2}}{|J_1|^{1/2}}\frac{|I_2|^{1/2}}{|J_2|^{1/2}}.
\]
\end{lem}

The proof of the lemma above is completely the same as Lemma \ref{oneparaBMO}, which is left to the reader.

\subsubsection{Term IV, VI, VII}

Now we turn to term IV (symmetric with term VI), which can be realized in a form involving bi-parameter paraproduct. Write
\[
\begin{split}
&\sum_{I_1\subsetneq J_1}\sum_{I_2\subsetneq J_2}\sum_{J_3\subsetneq I_3} \ave{h_{J_2}}_{I_2}\ave{h_{I_3}}_{J_3}\pair{T(h_{I_1}\otimes h_{I_2}\otimes 1)}{s_{I_1J_1}\otimes 1\otimes h_{J_3}}\cdot\\
&\qquad\qquad\pair{f}{h_{I_1}\otimes h_{I_2}\otimes h_{I_3}}\pair{g}{h_{J_1}\otimes h_{J_2}\otimes h_{J_3}}\\
&=\sum_{I_1\subsetneq J_1}\sum_{V}\sum_{W}\ave{\pair{g}{h_{J_1}\otimes h_W}_{1,3}}_V\ave{\pair{f}{h_{I_1}\otimes h_V}_{1,2}}_W\cdot\\
&\qquad\qquad\pair{T(h_{I_1}\otimes h_V\otimes 1)}{s_{I_1J_1}\otimes 1\otimes h_W},
\end{split}
\]
which is of the form
\[
C\sum_{i_1=1}^\infty 2^{-i_1\delta/2}\pair{S^{i_100000}f}{g},
\]
if one can prove that the following correct normalization holds true:
\[
|\pair{T(h_{I_1}\otimes h_V\otimes 1)}{s_{I_1J_1}\otimes 1\otimes h_W}|\lesssim \frac{|I_1|^{1/2}}{|J_1|^{1/2}}\left(\frac{\ell(I_1)}{\ell(J_1)}\right)^{\delta/2}|V|^{1/2}|W|^{1/2}.
\]

To see this, recall that by the partial kernel representation,
\[
\begin{split}
&\pair{T(h_{I_1}\otimes h_V\otimes 1)}{s_{I_1J_1}\otimes 1\otimes h_W}=\pair{T_2(h_{I_1}\otimes 1\otimes 1)}{s_{I_1J_1}\otimes h_V\otimes h_W}\\
&=\int_{\mathbb{R}^{d_1}}\int_{\mathbb{R}^{d_1}}K^{\{1\}}_{2,1\otimes 1,h_V\otimes h_W}(x_1,y_1)h_{I_1}(y_1)s_{I_1J_1}(x_1)\,dx_1dy_1,
\end{split}
\]
where the partial kernel $K^{\{1\}}_{2,1\otimes 1,h_V\otimes h_W}$ satisfies standard kernel estimates bounded by constant $C^{\{1\}}(1\otimes 1,h_V\otimes h_W)$, where additionally we have the assumption that $C^{\{1\}}(1\otimes 1,\cdot)$ is a function in $BMO_{prod}(\mathbb{R}^{d_2}\times\mathbb{R}^{d_3})$ with norm $\lesssim 1$. Hence, there holds $C^{\{1\}}(1\otimes 1,h_V\otimes h_W)\lesssim |V|^{1/2}|W|^{1/2}$, and the correct normalization of the coefficient then follows from a completely same argument as Lemma $3.10$ in \cite{Hy}.

It is then left to demonstrate the uniform boundedness of the shift $S^{i_100000}$. Rewrite
\[
\begin{split}
&\sum_{V}\sum_{W}\ave{\pair{g}{h_{J_1}\otimes h_W}_{1,3}}_V\ave{\pair{f}{h_{I_1}\otimes h_V}_{1,2}}_W\pair{T(h_{I_1}\otimes h_V\otimes 1)}{s_{I_1J_1}\otimes 1\otimes h_W}\\
&=C2^{-i_1\delta/2}\pair{h_{J_1}\otimes\Pi_{b_{I_1J_1}}(\pair{f}{h_{I_1}}_1)}{g},
\end{split}
\]
where $b_{I_1J_1}:=\pair{T_2(h_{I_1}\otimes 1\otimes 1)}{s_{I_1J_1}}_1/(C2^{-i_1\delta/2})$. The bi-parameter paraproduct
\[
\Pi_{b}(f):=\sum_{V,W}\pair{b}{h_V\otimes h_W}_{2,3}\pair{f}{h_V\otimes h_W^1}_{2,3}h_V^1\otimes h_W|V|^{-1/2}|W|^{-1/2},
\]
where $h_V^1,h_W^1$ are noncancellative Haar functions defined as $|V|^{-1/2}\chi_V, |W|^{-1/2}\chi_W$, respectively. Since the boundedness of $\Pi_{b_{I_1J_1}}$ implies the uniform boundedness of $S^{i_100000}$ similarly as in Proposition \ref{bdd}, it thus suffices to prove the following result:

\begin{lem}
$\|b_{I_1J_1}\|_{BMO_{prod}(\R^{d_2}\times\R^{d_3})}\lesssim\frac{|I_1|^{1/2}}{|J_1|^{1/2}}$.
\end{lem}

\begin{proof}
To see this, one needs to refer to the partial kernel assumption and the WBP/BMO conditions of the constant. More specifically, we will prove that for any dyadic grids $\mathcal{D}_2,\mathcal{D}_3$, and any open set $\Omega\subset\mathbb{R}^{d_2}\times\mathbb{R}^{d_3}$ with finite measure, there holds
\[
\frac{1}{|\Omega|}\sum_{\substack{R\subset\Omega, R\in\mathcal{D}_2\times\mathcal{D}_3\\R=J_2\times J_3}}|\pair{T_2(h_{I_1}\otimes 1\otimes 1)}{s_{I_1J_1}\otimes h_{J_2}\otimes h_{J_3}}|^2/(C^22^{-i_1\delta})\lesssim \frac{|I_1|}{|J_1|}.
\]

Due to the disjoint supports of $h_{I_1}$ and $s_{I_1J_1}$, one has
\begin{equation}\label{Carleson}
\pair{T_2(h_{I_1}\otimes 1\otimes 1)}{s_{I_1J_1}\otimes h_{J_2}\otimes h_{J_3}}=\int_{I_1}\int_{Q_1^c}K^{\{1\}}_{2,1\otimes 1,h_{J_2}\otimes h_{J_3}}(x_1,y_1)h_{I_1}(y_1)s_{I_1J_1}(x_1)\,dx_1dy_1.
\end{equation}

If $\ell(I_1)<2^{-r}\ell(J_1)$, the goodness of $I_1$ implies $d(I_1, Q_1^c)\geq\ell(J_1)(\ell(I_1)/\ell(J_1))^{\gamma_1}$. Hence, according to the mean zero property of $h_{I_1}$ and H\"older condition of the partial kernel, one has
\[
\begin{split}
&|(\ref{Carleson})|\\
&=|\int_{I_1}\int_{Q_1^c}[K^{\{1\}}_{2,1\otimes 1,h_{J_2}\otimes h_{J_3}}(x_1,y_1)-K^{\{1\}}_{2,1\otimes 1,h_{J_2}\otimes h_{J_3}}(x_1,c(I_1))]h_{I_1}(y_1)s_{I_1J_1}(x_1)\,dx_1dy_1|\\
&\lesssim C^{\{1\}}_2(1\otimes 1,h_{J_2}\otimes h_{J_3})\|h_{I_1}\|_1\|s_{I_1J_1}\|_{\infty}|\int_{Q_1^c}\frac{\ell(I_1)^\delta}{d(x_1,I_1)^{d_1+\delta}}\,dx_1|\\
&\lesssim C^{\{1\}}_2(1\otimes 1,h_{J_2}\otimes h_{J_3})\frac{|I_1|^{1/2}}{|J_1|^{1/2}}\left(\frac{\ell(I_1)}{\ell(J_1)}\right)^{\delta/2}.
\end{split}
\]

If $\ell(I_1)\geq 2^{-r}\ell(J_1)$ instead, we further split (\ref{Carleson}) into two parts. Write
\[
\begin{split}
&|(\ref{Carleson})|\\
&\leq\int_{3I_1\setminus I_1}|\int_{I_1}K^{\{1\}}_{2,1\otimes 1,h_{J_2}\otimes h_{J_3}}(x_1,y_1)h_{I_1}(y_1)\,dy_1||s_{I_1J_1}(x_1)|\,dx_1\\
&\, +\int_{(3I_1)^c}|\int_{I_1}[K^{\{1\}}_{2,1\otimes 1,h_{J_2}\otimes h_{J_3}}(x_1,y_1)-K^{\{1\}}_{2,1\otimes 1,h_{J_2}\otimes h_{J_3}}(x_1,c(I_1))]h_{I_1}(y_1)\,dy_1||s_{I_1J_1}(x_1)|\,dx_1\\
&\lesssim C^{\{1\}}_{2}(1\otimes 1, h_{J_2}\otimes h_{J_3})\|h_{I_1}\|_{\infty}\|s_{I_1J_1}\|_{\infty}\int_{3I_1\setminus I_1}\int_{I_1}\frac{1}{|x_1-y_1|^{d_1}}\,dy_1dx_1\\
&\quad+C^{\{1\}}_2(1\otimes 1,h_{J_2}\otimes h_{J_3})\|h_{I_1}\|_1\|s_{I_1J_1}\|_{\infty}\int_{(3I_1)^c}\frac{\ell(I_1)^\delta}{d(x_1,I_1)^{d_1+\delta}}\,dx_1\\
&\lesssim C^{\{1\}}_2(1\otimes 1, h_{J_2}\otimes h_{J_3})\frac{|I_1|^{1/2}}{|J_1|^{1/2}}\lesssim C^{\{1\}}_2(1\otimes 1, h_{J_2}\otimes h_{J_3})\frac{|I_1|^{1/2}}{|J_1|^{1/2}}\left(\frac{\ell(I_1)}{\ell(J_1)}\right)^{\delta/2}.
\end{split}
\]

Combining the two cases, we obtain
\[
|(\ref{Carleson})|\lesssim C^{\{1\}}_2(1\otimes 1,h_{J_2}\otimes h_{J_3})\frac{|I_1|^{1/2}}{|J_1|^{1/2}}2^{-i_1\delta/2},
\]
which then implies that
\[
\begin{split}
&\frac{1}{|\Omega|}\sum_{\substack{R\subset\Omega, R\in\mathcal{D}_2\times\mathcal{D}_3\\R=J_2\times J_3}}|\pair{T_2(h_{I_1}\otimes 1\otimes 1)}{s_{I_1J_1}\otimes h_{J_2}\otimes h_{J_3}}|^2/(C^22^{-i_1\delta})\\
&\quad\lesssim \frac{1}{|\Omega|}\sum_{\substack{R\subset\Omega, R\in\mathcal{D}_2\times\mathcal{D}_3\\R=J_2\times J_3}}|C^{\{1\}}_2(1\otimes 1,h_{J_2}\otimes h_{J_3})|^2\frac{|I_1|}{|J_1|}\lesssim \frac{|I_1|}{|J_1|},
\end{split}
\]
where the last step follows from the WBP/BMO assumption that $C^{\{1\}}_2(1\otimes 1,\cdot)$ is a product BMO function with norm $\lesssim 1$.

\end{proof}

This finishes the discussion the term IV. Similarly, term VII can also be organized as a sum of terms involving bi-parameter paraproducts, where the BMO function and the correct boundedness are given in the following lemma, whose proof is left to the reader.

\begin{lem}
Define $b_{J_3I_3}=\pair{T^*(1\otimes 1\otimes h_{J_3})}{s_{J_3I_3}}_3/(C2^{-j_3\delta/2})$, then,
\[
\|b_{J_3I_3}\|_{BMO_{prod}(\R^{d_1}\times\R^{d_2})}\lesssim \frac{|J_3|^{1/2}}{|I_3|^{1/2}}.
\]
\end{lem}

\subsubsection{Term VIII}

In order to deal with the last term, one needs to realize it into the desired form using tri-parameter paraproducts and apply the assumed mixed BMO/WBP conditions. More specifically, write
\[
\begin{split}
&\sum_{I_1\subsetneq J_1}\sum_{I_2\subsetneq J_2}\sum_{J_3\subsetneq I_3}\ave{h_{J_1}}_{I_1}\ave{h_{J_2}}_{I_2}\ave{h_{I_3}}_{J_3}\pair{T_3^*(1)}{h_{I_1}\otimes h_{I_2}\otimes h_{J_3}}\cdot\\
&\qquad\qquad\qquad\pair{f}{h_{I_1}\otimes h_{I_2}\otimes h_{I_3}}\pair{g}{h_{J_1}\otimes h_{J_2}\otimes h_{J_3}}\\
&=\sum_{K,V,W}\ave{\pair{g}{h_W}_3}_{K\times V}\ave{\pair{f}{h_K\otimes h_V}_{1,2}}_W\pair{T_3^*(1)}{h_K\otimes h_V\otimes h_W}\\
&=\sum_{K,V,W}\pair{T_3^*(1)}{h_K\otimes h_V\otimes h_W}\pair{f}{h_K\otimes h_V\otimes h_W^1}\cdot\\
&\qquad\qquad\qquad\pair{g}{h_K^1\otimes h_V^1\otimes h_W}|K|^{-1/2}|V|^{-1/2}|W|^{-1/2}\\
&=:\pair{\Pi_{T_3^*(1)}f}{g},
\end{split}
\]
where the tri-parameter paraproduct above is defined as
\[
\Pi_b(f):=\sum_{K,V,W}\pair{b}{h_K\otimes h_V\otimes h_W}\pair{f}{h_K\otimes h_V\otimes h_W^1}h_K^1\otimes h_V^1\otimes h_W|K|^{-1/2}|V|^{-1/2}|W|^{-1/2}.
\]

A hybrid square/maximal function argument shows that in the setting of arbitrarily many parameters, the analogue of paraproduct $\Pi_b$ defined above is always bounded on $L^2$ for product BMO symbol function $b$. Since it is one of our mixed BMO/WBP assumptions that $T_3^*(1)\in BMO$, term VIII can thus be realized of the form $C\pair{S^{000000}f}{g}$. And the proof of the case Inside/Inside/Inside is therefore complete.

Now one can see that for estimate of other cases where not all the pairs of cubes are nested, less multi-parameter paraproduct type estimates are involved. One just needs to carefully apply the various standard kernel assumptions to make things work, which shouldn't be hard once we've seen what is happening in this more difficult case. It is also not hard to observe that our argument can be easily adapted to handle all the different mixed cases due to the symmetry of our conditions formulated at the beginning of the paper, hence the proof of Theorem \ref{Repre} is complete.

Before ending the section, we emphasize that unlike \cite{Ma}, in the setting of more than two parameters, one has to deal with "partial type" multi-parameter paraproducts (for example for term IV, VI, VII above) in addition to the classical one-parameter ones in the discussion of  the above and other cases. This explains why one needs to formulate the full kernel, partial kernel, BMO/WBP assumptions for the operator $T$ in such a particular way as we did.

\section{Comparison to Journ\'e's class}\label{Journe}

The first general enough class of bi-parameter singular integral operators containing non-convolution type operators was established by Journ\'e in \cite{Jo}, where he proved a bi-parameter $T1$ theorem as well. It is also pointed out in \cite{Jo} that, by induction, his approach can be generalized to arbitrarily many parameters. 

\begin{defn}\label{jsio}
Let $T: C_0^\infty(\R^{d_1})\otimes C_0^\infty(\R^{d_2})\rightarrow [C_0^\infty(\R^{d_1})\otimes C_0^\infty(\R^{d_2})]'$ be a continuous linear mapping. It is a \emph{Journ\'e type bi-parameter $\delta$-SIO} if there exists a pair $(K_1, K_2)$ of $\delta CZ$-$\delta$-standard kernels so that, for all $f_1,g_1\in C_0^\infty(\R^{d_1})$ and $f_2,g_2\in C_0^\infty(\R^{d_2})$,
\begin{equation}
\pair{T(f_1\otimes f_2)}{g_1\otimes g_2}=\int f_1(y_1)\pair{K_1(x_1,y_1)f_2}{g_2}g_1(x_1)\,dx_1dy_1
\end{equation} 
when $\text{spt}f_1\cap\text{spt}g_1=\emptyset$;
\begin{equation}
\pair{T(f_1\otimes f_2)}{g_1\otimes g_2}=\int f_2(y_2)\pair{K_2(x_2,y_2)f_1}{g_1}g_2(x_2)\,dx_2dy_2
\end{equation} 
when $\text{spt}f_2\cap\text{spt}g_2=\emptyset$.
\end{defn}

Recall that a $\delta CZ$-$\delta$-standard kernel is a standard kernel with parameter $\delta$ whose value is in the Banach space $\delta CZ$, the space of Calder\'on-Zygmund operators equipped with norm $\|T\|_{L^2\rightarrow L^2}+\|K\|$.

Let $T_1$ denote the partial adjoint $T_S$ where $S=\{1\}$, then it is easy to see that $T_1$ is also a Journ\'e type $\delta$-SIO if $T$ is. And a Journ\'e type $\delta$-SIO $T$ is called a \emph{Journ\'e type bi-parameter $\delta$-CZO} if both $T, T_1$ are bounded on $L^2$, associated with the norm $\|T\|_{L^2\rightarrow L^2}+\|T_1\|_{L^2\rightarrow L^2}+\|K_1\|_{\delta CZ}+\|K_2\|_{\delta CZ}$.

By induction, one can define \emph{Journ\'e type $n$-parameter SIO} accordingly. 

It is recently proved by Grau de la Herran in \cite{Gr} that in the bi-parameter setting, under the additional assumption that $T$ is bounded on $L^2$, $T$ is a Journ\'e type $\delta$-SIO satisfying certain WBP if and only if it satisfies Martikainen's mixed type conditions in \cite{Ma}. In the following, we reformulate this theorem without any assumption of the $L^2$ boundedness and prove it in the multi-parameter setting. In \cite{Gr}, the $L^2$ boundedness is used only to compare the two different formulations of WBP. However, in both Journ\'e's and our class of singular integrals, the WBP enter only in the context of the boundedness of the operator.

In the proof of Theorem \ref{equiv}, one of the intrinsic new difficulties is that some type of multi-parameter $T1$ theorem is needed, namely Corollary \ref{Cor}. The main theorem of this section is the following:

\begin{thm}\label{equiv}
$T$ is an $n$-parameter singular integral operator satisfying both the full kernel and partial kernel assumptions (see section \ref{full}, \ref{partial}) if and only if it is a Journ\'e type $n$-parameter SIO (see Definition \ref{jsio}).
\end{thm}

\begin{proof}
We will prove the case $n=3$ as an example, which is enough to show the new multi-parameter phenomena in the problem. And for simplicity of notations, let's assume that the dimensions $d_1=d_2=d_3=1$. To remind ourselves, $T$ is a Journ\'e type tri-parameter SIO if there exists a triple $(K_1, K_2, K_3)$ of $\delta CZ(\R\times\R)$-$\delta$-standard kernels such that
\begin{equation}
\pair{T(f_1\otimes f_2\otimes f_3)}{g_1\otimes g_2\otimes g_3}=\int f_1(y_1)\pair{K_1(x_1,y_1)f_2\otimes f_3}{g_2\otimes g_3}g_1(x_1)\,dx_1dy_1
\end{equation}
when $\text{spt}f_1\cap\text{spt}g_1=\emptyset$, and similarly for $K_2, K_3$. 

It is important to keep in mind that for any fixed $x_1, y_1$, $K_1(x_1,y_1)$ is a Journ\'e type bi-parameter SIO on $\R\times\R$.

To show that any Journ\'e type tri-parameter SIO $T$ satisfies our full and partial kernel assumptions, one can basically follow the strategy in \cite{Gr}, and note that no $L^2$ boundedness is needed. Due to the symmetries of the conditions, it suffices to check the kernel assumptions for $T$ while the results for other $T_S$ follow similarly. The full kernel assumptions are straightforward to verify, which we omit. For the partial kernel assumptions, let's look at the most difficult case $V=\{1\}$ as an example, while all the other cases follow similarly and symmetrically.

For any $\text{spt}f_1\cap\text{spt}g_1=\emptyset$, since $T$ is a Journ\'e type operator, we have
\[
\pair{T(f_1\otimes f_2\otimes f_3)}{g_1\otimes g_2\otimes g_3}=\int f_1(y_1)\pair{K_1(x_1,y_1)f_2\otimes f_3}{g_2\otimes g_3} g_1(x_1)\,dx_1dy_1.
\]
Define partial kernel $K^{\{1\}}_{f_2\otimes f_3,g_2\otimes g_3}(x_1,y_1):=\pair{K_1(x_1,y_1)f_2\otimes f_3}{g_2\otimes g_3}$. Then the mixed size-H\"older conditions are implied by the fact that $K_1(x_1,y_1)$ is a $\delta CZ(\R\times\R)$-$\delta$-standard kernel. Let's first look at the standard kernel estimates and the boundedness of constant $C^{\{1\}}(1\otimes 1,\cdot)$. Since $K_1(x_1,y_1)$ maps $L^\infty(\R\times\R)$ boundedly into $BMO_{prod}(\R\times\R)$ with an operator norm bounded by $\|K_1(x_1,y_1)\|_{\delta CZ(\mathbb{R}\times\mathbb{R})}$, a result proved by Journ\'e in \cite{Jo}. $K^{\{1\}}_{1\otimes 1,g_{23}}$ is thus well defined for any function $g_{23}\in H^1(\mathbb{R}\times\mathbb{R})$, not necessarily to be a tensor product.

Then in order to prove the size condition, one writes
\[
|K^{\{1\}}_{1\otimes 1,g_{23}}(x_1,y_1)|=|\pair{K_1(x_1,y_1)1\otimes 1}{g_{23}}|\lesssim\|K_1(x_1,y_1)\|_{\delta CZ(\R\times\R)},
\]
where $\|g_{23}\|_{H^1(\R\times\R)}\leq 1$. Hence, by the vector-valued standard kernel assumption of $K_1(x_1,y_1)$,
\[
|K^{\{1\}}_{1\otimes 1,g_{23}}(x_1,y_1)|\leq C^{\{1\}}(1\otimes 1,g_{23})\frac{1}{|x_1-y_1|},
\]
where $C^{\{1\}}(1\otimes 1,g_{23})$ is some constant universally bounded.

For H\"older conditions, one can similarly write
\[
\begin{split}
&|K^{\{1\}}_{1\otimes 1,g_{23}}(x_1,y_1)-K^{\{1\}}_{1\otimes 1,g_{23}}(x_1',y_1)|=|\pair{(K_1(x_1,y_1)-K_1(x_1',y_1))1\otimes 1}{g_{23}}|\\
&\lesssim \|K_1(x_1,y_1)-K_1(x_1',y_1)\|_{\delta CZ(\mathbb{R}\times\mathbb{R})}\lesssim C^{\{1\}}(1\otimes 1,g_{23})\frac{|x_1-x_1'|^{\delta}}{|x_1-y_1|^{1+\delta}},
\end{split}
\]
where the constant $C^{\{1\}}(1\otimes 1,g_{23})$ is the same as before. This completes the proof of the standard kernel estimates and the BMO condition of $C^{\{1\}}(1\otimes 1,\cdot)$ as well.

To prove the bounds for $C^{\{1\}}(\chi_{I_2}\otimes 1,\chi_{I_2}\otimes h)$ ($h$ being an atom of $H^1(\mathbb{R})$ adapted to cube $V$), for simplicity we only verify the size condition as the H\"older conditions are similar. Split
\[
\begin{split}
&K^{\{1\}}_{\chi_{I_2}\otimes 1,\chi_{I_2}\otimes h}(x_1,y_1)=\pair{K_1(x_1,y_1)\chi_{I_2}\otimes 1}{\chi_{I_2}\otimes h}\\
&=\pair{K_1(x_1,y_1)\chi_{I_2}\otimes \chi_{3V}}{\chi_{I_2}\otimes h}+\pair{K_1(x_1,y_1)\chi_{I_2}\otimes \chi_{(3V)^c}}{\chi_{I_2}\otimes h}=: I+II.
\end{split}
\] 

The first term can be estimated using $L^2$ bounds:
\[
|I|\leq\|K_1(x_1,y_1)\|_{\delta CZ(\mathbb{R}\times\mathbb{R})}\|\chi_{I_2}\otimes\chi_{3V}\|_2\|\chi_{I_2}\otimes h\|_2\lesssim\|K_1(x_1,y_1)\|_{\delta CZ(\mathbb{R}\times\mathbb{R})}|I_2|.
\]

For the second term, noticing that $\chi_{(3V)^c}$ and $h$ are disjointly supported, by the definition of bi-parameter Journ\'e type CZO, there exists Calder\'on-Zygmund operator $K_1^3(x_1,y_1,x_3,y_3)$ such that
\[
II=\int\chi_{(3V)^c}(y_3)\pair{K_1^3(x_1,y_1,x_3,y_3)\chi_{I_2}}{\chi_{I_2}}h(x_3)\,dx_3dy_3,
\]
which by the vector-valued standard kernel estimate equals
\[
\begin{split}
&=\int\chi_{(3V)^c}(y_3)\pair{[K_1^3(x_1,y_1,x_3,y_3)-K_1^3(x_1,y_1,x_3,c(V))]\chi_{I_2}}{\chi_{I_2}}h(x_3)\,dx_3dy_3\\
&\leq |I_2|\int|\chi_{(3V)^c}(y_3)h(x_3)|\|K_1^3(x_1,y_1,x_3,y_3)-K_1^3(x_1,y_1,x_3,c(V))\|_{\delta CZ(\mathbb{R})}\,dx_3dy_3\\
&\leq |I_2|\|K_1(x_1,y_1)\|_{\delta CZ(\mathbb{R}\times\mathbb{R})}\int|\chi_{(3V)^c}(y_3)h(x_3)|\frac{\ell(V)^\delta}{d(y_3,V)^{1+\delta}}\,dx_3dy_3\\
&\lesssim |I_2|\|K_1(x_1,y_1)\|_{\delta CZ(\mathbb{R}\times\mathbb{R})}.
\end{split}
\]

One thus has the size condition
\[
|K^{\{1\}}_{\chi_{I_2}\otimes 1,\chi_{I_2}\otimes h}(x_1,y_1)|\lesssim C^{\{1\}}(\chi_{I_2}\otimes 1,\chi_{I_2}\otimes h)\frac{1}{|x_1-y_1|},
\]
where the constant is taken so that $C^{\{1\}}(\chi_{I_2}\otimes 1,\chi_{I_2}\otimes h)\lesssim |I_2|$, hence satisfies the desired BMO estimate.

Lastly, the estimate of $C^{\{1\}}(\chi_{I_2}\otimes\chi_{I_3},\chi_{I_2}\otimes\chi_{I_3})$ can be proved similarly solely based on the $L^2$ boundedness of $K_1(x_1,y_1)$, which completes the easy direction of the proof of Theorem \ref{equiv}.

To prove the other direction, for any given tri-parameter operator $T$, together with all of its partial adjoints satisfying the full and partial kernel assumptions, we are going to prove that it is a Journ\'e type SIO, i.e. there exist $\delta CZ(\R\times\R)$-$\delta$-standard kernels $K_1,K_2$ and $K_3$. By symmetry, it suffices to show the existence of $K_1$.

For any $\text{spt}f_1\cap\text{spt}g_1=\emptyset$, there holds for some partial kernel $K^{\{1\}}_{f_2\otimes f_3,g_2\otimes g_3}$ that
\[
\pair{T(f_1\otimes f_2\otimes f_3)}{g_1\otimes g_2\otimes g_3}=\int K^{\{1\}}_{f_2\otimes f_3,g_2\otimes g_3}(x_1,y_1)f_1(y_1)g_1(x_1)\,dx_1dy_1.
\]
This suggests us to define a bi-parameter operator $K_1(x_1,y_1)$ associated with the following bilinear form:
\[
\pair{K_1(x_1,y_1)f_2\otimes f_3}{g_2\otimes g_3}:=K^{\{1\}}_{f_2\otimes f_3,g_2\otimes g_3}(x_1,y_1).
\]
It is left to prove that $K_1(x_1,y_1)$ is a Journ\'e type $\delta$-CZO on $\R\times\R$ and satisfies the standard kernel estimates. For the sake of brevity, we will focus only on the size condition, i.e. to show that $\|K_1(x_1,y_1)\|_{\delta CZ(\R\times\R)}\lesssim |x_1-y_1|^{-1}$.

For any fixed $x_1,y_1$, the fact that $K_1(x_1,y_1)$ defined above is indeed a linear continuous mapping follows from the linearity and continuity of $T$ itself, with the aid of Lebesgue differentiation theorem. 

To see that $K_1(x_1,y_1)$ is a Journ\'e type bi-parameter $\delta$-SIO, according to the definition, we need to show the existence of a pair $(K^2_1(x_1,y_1, x_2,y_2),K^3_1(x_1,y_1,x_3,y_3))$ of $\delta CZ$-$\delta$-standard kernels such that
\begin{equation}
\begin{split}
&\pair{K_1(x_1,y_1)f_2\otimes f_3}{g_2\otimes g_3}=K^{\{1\}}_{f_2\otimes f_3,g_2\otimes g_3}(x_1,y_1)\\
&=\int f_2(y_2)\pair{K_1^2(x_1,y_1,x_2,y_2)f_3}{g_3}g_2(x_2)\,dx_2dy_2
\end{split}
\end{equation}
when $\text{spt}f_2\cap\text{spt}g_2=\emptyset$;
\begin{equation}
\begin{split}
&\pair{K_1(x_1,y_1)f_2\otimes f_3}{g_2\otimes g_3}=K^{\{1\}}_{f_2\otimes f_3,g_2\otimes g_3}(x_1,y_1)\\
&=\int f_3(y_3)\pair{K_1^3(x_1,y_1,x_3,y_3)f_2}{g_2}g_3(x_3)\,dx_3dy_3
\end{split}
\end{equation}
when $\text{spt}f_3\cap\text{spt}g_3=\emptyset$, and the bound $\|K_1^i(x_1,y_1,x_i,y_i)\|_{\delta CZ}\lesssim |x_1-y_1|^{-1}$ for $i=2,3$.

The existence of $K_1^2,K_1^3$ follows from another partial kernel assumption. Let's take $K_1^2$ as an example. When $\text{spt}f_i\cap\text{spt}g_i=\emptyset$ for $i=1,2$,
\[
\begin{split}
&\pair{T(f_1\otimes f_2\otimes f_3)}{g_1\otimes g_2\otimes g_3}\\
&=\int K^{\{1,2\}}_{f_3,g_3}(x_1,y_1,x_2,y_2)f_1(y_1)f_2(y_2)g_1(x_1)g_2(x_2)\,dx_1dx_2dy_1dy_2\\
&=\int K^{\{1\}}_{f_2\otimes f_3,g_2\otimes g_3}(x_1,y_1)f_1(y_1)g_1(x_1)\,dx_1dy_1.
\end{split}
\]

By Lebesgue differentiation, this implies
\[
\begin{split}
&\pair{K_1(x_1,y_1)f_2\otimes f_3}{g_2\otimes g_3}=K^{\{1\}}_{f_2\otimes f_3,g_2\otimes g_3}(x_1,y_1)\\
&=\int K^{\{1,2\}}_{f_3,g_3}(x_1,y_1,x_2,y_2)f_2(y_2)g_2(x_2)\,dx_2dy_2.
\end{split}
\]
It thus natural to define $\pair{K_1^2(x_1,y_1,x_2,y_2)f_3}{g_3}:=K^{\{1,2\}}_{f_3,g_3}(x_1,y_1,x_2,y_2)$.

We next prove $\|K_1^2(x_1,y_1,x_2,y_2)\|_{\delta CZ}\lesssim |x_1-y_1|^{-1}|x_2-y_2|^{-1}$, which is the pure size estimate, and the mixed size-H\"older estimates follow similarly.

First, one can easily check that operator $K_1^2(x_1,y_1,x_2,y_2)$ is associated with the kernel $K(x_1,y_2,x_2,y_2,\cdot,\cdot)$, which is standard with the correct norm because of the mixed size-H\"older conditions in the full kernel assumption. It thus suffices to prove that $\|K_1^2(x_1,y_1,x_2,y_2)\|_{L^2\rightarrow L^2}\lesssim |x_1-y_1|^{-1}|x_2-y_2|^{-1}$, which will follow from Corollary \ref{Cor} in the case $n=1$ provided that $K_1^2(x_1,y_1,x_2,y_2)$ satisfies the BMO/WBP properties. (This is exactly the classical $T1$ theorem, rephrased in our setting.)

To see this last piece of fact, note that for any normalized $H^1$ function $h$, any cube $I_3$ in the third variable,
\[
\begin{split}
|\pair{K_1^2(x_1,y_1,x_2,y_2)1}{h}|=|K^{\{1,2\}}_{1,h}(x_1,y_1,x_2,y_2)|&\lesssim C^{\{1,2\}}(1,h)\frac{1}{|x_1-y_1|}\frac{1}{|x_2-y_2|}\\&\lesssim \frac{1}{|x_1-y_1|}\frac{1}{|x_2-y_2|},
\end{split}
\]
and
\[
\begin{split}
&|\pair{K_1^2(x_1,y_1,x_2,y_2)\chi_{I_3}}{\chi_{I_3}}|=|K^{\{1,2\}}_{\chi_{I_3},\chi_{I_3}}(x_1,y_1,x_2,y_2)|\\
&\lesssim C^{\{1,2\}}(\chi_{I_3},\chi_{I_3})\frac{1}{|x_1-y_1|}\frac{1}{|x_2-y_2|}\lesssim |I_3|\frac{1}{|x_1-y_1|}\frac{1}{|x_2-y_2|},
\end{split}
\]
which are the BMO/WBP assumptions when $n=1$. This demonstrates that $K_1(x_1,y_1)$ is a Journ\'e type bi-parameter $\delta$-SIO on $\R\times \R$.

Now the only gap left in the proof of Theorem \ref{equiv} is to show that as a bi-parameter operator, 
\begin{equation}\label{multi}
\|K_1(x_1,y_1)\|_{L^2\rightarrow L^2}\lesssim\frac{1}{|x_1-y_1|},
\end{equation}
together with the same bound for its partial adjoint. We omit the proof of the partial adjoint part as it follows from the same argument by changing $T$ to its corresponding partial adjoint from the beginning.

The proof of (\ref{multi}) is exactly where the multi-parameter version of Corollary \ref{Cor} comes into play, as we are in need of a multi-parameter $T1$ type theorem of its full strength. It thus suffices to demonstrate that $K_1(x_1,y_1)$ is a bi-parameter singular integral satisfying our full and partial kernel assumptions, as well as the additional BMO/WBP assumptions with the required norm. Note that without loss of generality, we are free to discuss $K_1(x_1,y_1)$ itself only, as the similar results for its partial adjoints will follow from the symmetry of the assumptions on $T$.

To demonstrate the full kernel assumption, noticing that $K_1(x_1,y_1)$ is associated with kernel $K(x_1,y_1,\cdot,\cdot,\cdot,\cdot)$, then it's not hard to check all the mixed size-H\"older conditions of the kernel.

For the partial kernel assumption, when $\text{spt}f_2\cap\text{spt}g_2=\emptyset$, observe that
\[
\pair{K_1(x_1,y_1)f_2\otimes f_3}{g_2\otimes g_3}=\int K^{\{1,2\}}_{f_3,g_3}(x_1,y_1,x_2,y_2)f_2(y_2)g_2(x_2)\,dx_2dy_2.
\]
Then, the partial kernel $K^{\{1,2\}}_{f_3,g_3}$ satisfies the collection of mixed size-H\"older conditions with a constant bounded by $C^{\{1,2\}}(f_3,g_3)|x_1-y_1|^{-1}$. And for any normalized $H^1$ function $h$ and any cube $I_3$,
\[
C^{\{1,2\}}(1,h)\lesssim 1,\quad C^{\{1,2\}}(\chi_{I_3},\chi_{I_3})\lesssim |I_3|.
\]
The H\"older estimate for the partial kernel follows similarly.

It's thus left to check the BMO/WBP assumptions. This will also follow from the partial kernel assumptions of $T$. First, for any dyadic grids $\mathcal{D}_2,\mathcal{D}_3$ and open set $\Omega\subset\mathbb{R}\times\mathbb{R}$ with finite measure, since
\[
|\pair{K_1(x_1,y_1)1\otimes 1}{h_{J_2}\otimes h_{J_3}}|=|K^{\{1\}}_{1\otimes 1,h_{J_2}\otimes h_{J_3}}(x_1,y_1)|\lesssim C^{\{1\}}(1\otimes 1,h_{J_2}\otimes h_{J_3})\frac{1}{|x_1-y_1|},
\]
there holds
\[
\begin{split}
&\frac{1}{|\Omega|}\sum_{\substack{R\subset\Omega,R\in\mathcal{D}_2\times\mathcal{D}_3\\R=J_2\times J_3}}|\pair{K_1(x_1,y_1)1\otimes 1}{h_{J_2}\otimes h_{J_3}}|^2\\
&\lesssim \frac{1}{|x_1-y_1|}\frac{1}{|\Omega|}\sum_{\substack{R\subset\Omega,R\in\mathcal{D}_2\times\mathcal{D}_3\\R=J_2\times J_3}}|C^{\{1\}}(1\otimes 1,h_{J_2}\otimes h_{J_3})|^2\\
&\lesssim \frac{1}{|x_1-y_1|}.
\end{split}
\]

The last inequality above follows from the fact that $C^{\{1\}}(1\otimes 1)$ is a product BMO function with norm $\lesssim 1$.
To verify that other BMO/WBP assumptions hold true, for any normalized $H^1(\R)$ function $h_3$ and cubes $I_2,I_3$, in the second and third variable respectively, observe that
\[
|\pair{K_1(x_1,y_1)\chi_{I_2}\otimes\chi_{I_3}}{\chi_{I_2}\otimes\chi_{I_3}}|\lesssim C^{\{1\}}(\chi_{I_2}\otimes\chi_{I_3},\chi_{I_2}\otimes\chi_{I_3})\frac{1}{|x_1-y_1|}\lesssim |I_2||I_3|\frac{1}{|x_1-y_1|},
\]
\[
|\pair{K_1(x_1,y_1)\chi_{I_2}\otimes 1}{\chi_{I_2}\otimes h_3}|\lesssim C^{\{1\}}(\chi_{I_2}\otimes 1,\chi_{I_2}\otimes h_3)\frac{1}{|x_1-y_1|}\lesssim |I_2|\frac{1}{|x_1-y_1|}.
\]

Hence, applying Corollary \ref{Cor} in the case $n=2$ will complete the proof.
\end{proof}

\begin{rem}
Note that when the number of parameters goes up, in order to prove Theorem \ref{equiv}, we have to use Corollary \ref{Cor} for arbitrarily many parameters, which is one of the applications of our $n$-parameter representation theorem for $n\geq 3$.
\end{rem}

Once we have Theorem \ref{equiv}, it is natural to obtain the following characterization of Journ\'e type $n$-parameter $\delta$-CZO as well.

\begin{cor}
$T$ is a Journ\'e type $n$-parameter $\delta$-CZO if and only if it is an $n$-parameter CZO defined in Section \ref{Assump}. 
\end{cor}

\begin{proof}
As we have shown in Theorem \ref{equiv} that Journ\'e's and our classes of $n$-parameter SIO are equivalent. It is thus left to verify the equivalence between boundedness of all the partial adjoints of $T$. This can be shown directly from the inductive definition of Journ\'e type $n$-parameter CZO, observing that in $(n-1)$-parameter, the partial kernels are always CZOs themselves, satisfying the corresponding $L^2$ boundedness in $(n-1)$-parameter.
\end{proof}

Up to this point, we have successfully established a set of characterizing conditions for an operator to be a Journ\'e type $n$-parameter CZO. This is very useful in the study of multi-parameter operators since the full kernel, partial kernel, BMO/WBP conditions are usually much easier to verify and used compared with Journ\'e's original vector-valued formulation.

\section{Some discussion of the necessity of the BMO/WBP conditions}
Given an $n$-parameter singular integral operator $T$ satisfying both full and partial kernel assumptions, one might ask if the mixed BMO/WBP conditions are necessary for $T$ to be bounded on $L^2(\mathbb{R}^{\vec{d}})$. The answer is yes when $n=1$, which is a classical result of Calder\'on-Zygmund operators, but is no for $n\geq 2$. In fact, a counterexample has been constructed in \cite{Jo} to show that in the bi-parameter setting, $T_1 1$ and $T_1^* 1\in BMO$ are not necessary conditions for $T$ to be bounded on $L^2$. 

However, one can indeed prove the necessity of some of the mixed BMO/WBP conditions, more specifically, those that are formulated for $T$ and $T^*$. It is straightforward to verify that pure WBP, i.e.
\[
|\pair{T(\chi_{I_1}\otimes\cdots\otimes\chi_{I_n})}{\chi_{I_1}\otimes\cdots\otimes\chi_{I_n}}|\lesssim \prod_{i=1}^n|I_i|
\]
is directly implied by the $L^2$ boundedness of $T$. For the pure BMO conditions: $T1, T^*1\in BMO$, the necessity is first pointed out in \cite{Gr} for bi-parameters, and is not hard to extend to arbitrarily many parameters using Theorem \ref{equiv}. To see this, suppose that there is a $L^2$ bounded $n$-parameter SIO satisfying full and partial kernel assumptions. By Theorem \ref{equiv}, $T$ is also a Journ\'e type $n$-parameter SIO who is bounded on $L^2$. Hence, Theorem $3$ in \cite{Jo} implies that $T1\in BMO$, as well as $T^* 1\in BMO$ taking into account that $T^*$ is also $L^2$ bounded.

To prove that for operator $T$ given above, there also hold the mixed BMO/WBP conditions for $T, T^*$, we take a look at the tri-parameter, $d_1=d_2=d_3=1$ case as an example. In other words, one wants to show that
\begin{equation}\label{necess}
\|\pair{T(\chi_{I_1}\otimes 1\otimes 1)}{\chi_{I_1}\otimes\cdot}\|_{BMO(\mathbb{R}\times\mathbb{R})}\lesssim |I_1|,
\end{equation}
\begin{equation}\label{necess2}
\|\pair{T(\chi_{I_1}\otimes\chi_{I_2}\otimes 1)}{\chi_{I_1}\otimes\chi_{I_2}\otimes\cdot}\|_{BMO(\mathbb{R})}\lesssim |I_1||I_2|,
\end{equation}
and all the other mixed BMO/WBP conditions formulated for $T$ will follow symmetrically, so are the ones for $T^*$.

In order to prove (\ref{necess}), for any cube $I_1$, one can define an operator $\pair{T^1\chi_{I_1}}{\chi_{I_1}}$ mapping $C_0^\infty(\mathbb{R})\otimes C_0^\infty(\mathbb{R})$ to its dual:
\[
\pair{\pair{T^1\chi_{I_1}}{\chi_{I_1}}f_2\otimes f_3}{g_2\otimes g_3}:=\pair{T(\chi_{I_1}\otimes f_2\otimes f_3)}{\chi_{I_1}\otimes g_2\otimes g_3}.
\]
By taking one parameter away, it is easy to see that $\pair{T^1\chi_{I_1}}{\chi_{I_1}}$ is a bi-parameter SIO, whose full kernel is $K^{\{2,3\}}_{\chi_{I_1},\chi_{I_1}}(x_2,x_3,y_2,y_3)$ with norm bounded by
\[
C^{\{2,3\}}(\chi_{I_1},\chi_{I_1})\lesssim |I_1|,
\]
while the partial kernel assumptions can be verified similarly. Moreover, following from the definition of $\pair{T^1\chi_{I_1}}{\chi_{I_1}}$ and the $L^2$ boundedness of $T$, one can conclude that $\pair{T^1\chi_{I_1}}{\chi_{I_1}}$ is a $L^2$ bounded bi-parameter Journ\'e type SIO with norm $\lesssim |I_1|$, thus maps $1\otimes 1$ boundedly into $BMO(\mathbb{R}\times\mathbb{R})$, which proves (\ref{necess}).

Using the same strategy, it is not hard to demonstrate $(\ref{necess2})$ by slicing two parameters away and apply the $L^\infty\rightarrow BMO$ estimate for Calder\'on-Zygmund operators. We omit the details.

This, together with the discussion at the end of section \ref{theorem}, leads us to the following characterizing result of the class of $n$-parameter CZO.

\begin{cor}
Given an $n$-parameter singular integral operator $T$ satisfying both full and partial kernel assumptions, it is then an $n$-parameter CZO if and only if the mixed BMO/WBP assumptions hold true.
\end{cor}

To end the paper, we state the following result and sketch the proof, which indicates the generality of our operator class and its inductive intrincity. Moreover, it also shows that although our class of operators has been proven to be equivalent to Journ\'e's, its mixed type characterizing conditions still provide us with a very helpful tool to study $n$-parameter operators, especially when $n$ is very large.

\begin{prop}
Let $T:=T_1\otimes T_2\otimes\cdots\otimes T_s$ be an operator on $\mathbb{R}^{\vec{d}}:=\mathbb{R}^{\vec{d_1}}\times\cdots\times\mathbb{R}^{\vec{d_s}}$, where for any $1\leq i\leq s$, $T_i$ is a $t_i$-parameter CZO on $\mathbb{R}^{\vec{d_i}}:=\mathbb{R}^{d_i^1}\times\cdots\times\mathbb{R}^{d_i^{t_i}}$. Then $T$ is an $n$-parameter CZO, where $n:=t_1+\cdots +t_s$.
\end{prop}

\begin{proof}
Observing that the partial adjoints of $T$ can be expressed as tensor products of some partial adjoints of $T_i$, it suffices to prove that $T$ itself verifies the full and partial kernel assumptions, as the $L^2$ boundedness is straightforward.

The full kernel assumption is easy to see, since the tensor product of kernels of $T_i$ is the full kernel and satisfies all the mixed size-H\"older conditions.

To show the partial kernel assumptions, note that in any case, one can always write the partial kernel as a tensor product of some of the full or partial kernels of $T_i$. And the BMO condition for the constants follow from the fact that the tensor product of partial kernels are always CZO with less parameters, hence maps $L^\infty\rightarrow BMO$. To prove the mixed WBP/BMO conditions for the constants, one just needs to take away more parameters and mimic what we did in the proof of (\ref{necess}) earlier this section. We leave the details of the proof to the readers.
\end{proof}

\end{document}